\documentclass[A4paper,20pt]{amsart}
\usepackage{graphicx} 
\usepackage[all,poly]{xy}
\usepackage{amsmath,amsthm,amsfonts,latexsym,amssymb, mathrsfs,stmaryrd,graphics,pst-all,tkz-euclide,yfonts,mathtools,tikz,graphicx,wrapfig}
\usepackage[utf8]{inputenc}
\usepackage{tikz-cd}
\usetikzlibrary{positioning}
\usepackage{fullpage}
\usepackage{color}
\usepackage[pagebackref,colorlinks]{hyperref} 
\usepackage{enumerate}
\scrollmode

\theoremstyle{plain}

\newtheorem{theorem}{Theorem}[section]
\newtheorem{lemma}[theorem]{Lemma}
\newtheorem{corollary}[theorem]{Corollary}
\newtheorem{proposition}[theorem]{Proposition}

\theoremstyle{definition}
\newtheorem{definition}[theorem]{Definition}

\newtheorem{remark}[theorem]{Remark}
\newtheorem{remarks and examples}[theorem]{Remarks and Examples}

\title{IT0 sheaves on compactified Jacobian}
\author{Pabitra Barik,Anindya Mukherjee}
\date{March 2026}

\begin{document}

\maketitle
\begin{abstract}
In this paper we construct stable $IT_{0}$ trosion free sheaves on compactified Jacobian of an integral nodal curve. Our construction is functorial in nature and it provides an uniform method which works for both smooth and nodal curve.
\end{abstract}

\section{Introduction}
Let $X$ be an abelian variety of dimension $g$ over algebrically closed field $k$ and denote $\hat{X}$ be the dual of $X$. Let $P$ be Poincare bundle on $X \times \hat{X}$ and for any coherent sheaf consider the Fourier-Mukai functor,\\
\begin{align*}
 \Phi_{P}: & D^{b}(X) \rightarrow D^{b}(\hat{X}) \\
           & F \rightarrow Rp_{2,*}(p_{1}^{*}F \otimes P)
\end{align*}

where $p_1$ and $p_2$ are two projections. Now let $C$ be smooth projective curve, $J(C)$ be the Jacobian and let $a:C \rightarrow J(C)$ be the usual Abel-Jacobi embedding. In \cite{Popa} $IT_{0}$ bundle on Jacobian has been constructed which arise using Fourier-Mukai functor from line bundles on curve. This $IT_{0}$ bundles satisfy Mukai regularity on abelian varieties  which is a stronger version of Mumford regularity. In this paper we construct stable $IT_{0}$ torsion free sheaves on Compactified Jacobian  which arise functorially from stable torsion free sheaves on integral nodal curve with one node. We expect that the existence of such stable torsion free sheaves of fixed rank may be helpful to study moduli of such objects.

\section{preliminaries}
Let $X$ be an integral nodal curve with one node. We denote the corresponding compactified jacobian $\overline {J(X)}$ by $A$. Following \cite{Coelho} we define the Abel-Jacobi embedding,\\

\subsection{\textbf{The Abel-Jacobi embedding}}
Fix  $L$ be a line bundle of degree $1$ on $X$. Then the Abel-Jacobi map is,\\
\begin{align*}
   a :& X \rightarrow \overline J(X),\\
      & p \rightarrow I_{p} \otimes L
\end{align*}

we identify the image of the nodal curve $X$ inside the compactified Jacobian $\overline {J(X)}$ via the closed immersion.

\subsection{\textbf{Fourier-Mukai transform on compactified Jacobian}}
In this section we define an analogue of Fourier-Mukai transform for compactified jacobian following \cite{Dima}.\\
Let $p_1:\overline J \times \overline J \rightarrow \overline J$ and $p_2:\overline J \times \overline J \rightarrow \overline J$ are two projections. We have the existence of universal Poincare sheaf $\overline P$ on $\overline J \times \overline J$\cite[Theorem A]{Dima}. So define the Fourier-Mukai functor \\
$\phi_{\overline P}:D^{b}(A) \rightarrow D^{b}(A)$ by $\phi_{\overline P}(F):=Rp_{1*}(p_{2}^{*}F \otimes \overline P)$. In \cite[Theorem C]{Dima} it is shown that $\phi_{\overline P}$ is a categorical equivalence.
\subsection{\textbf{WIT and IT conditions}}
Let $F$ be a coherent sheaf on $A$.
\begin{definition}
We say that $F$ satisfies $WIT_{i}$ if $R^{j}\Phi_{\overline P}(F)=0$ for all $j \neq i$.
\end{definition}
if $F$ satisfies $WIT_{i}$, we denote\\
\begin{equation*}
   \hat{F}=R^{i}\phi_{\overline P}(F)
\end{equation*}

\begin{definition}
    A coherent sheaf $F$ on $A$ satisfies $IT_{0}$ if \\
    \begin{equation*}
        H^{i}(A, F \otimes \alpha)=0, \forall \alpha \in \overline {Pic^{0}(\overline J)}
    \end{equation*}
    $IT_{0}$ is stronger than $WIT_{0}$.
\end{definition}
Now we define the Fourier-Mukai base change analogous to the smooth case. We have that the Poincare sheaf $\overline P$ is flat for the second projection $p_{2}:\overline J \times \overline J \rightarrow \overline J$ and also we have $\overline P|_{\overline J \times \alpha}$ is Cohen-Maculay see\cite{Dima}. Thus like smooth case we have also the derived basechange for Fourier-Mukai in this case.
\subsection{ Base change for Fourier-Mukai}
Let $F$ be a coherent sheaf on $A$. For $\alpha \in \overline{Pic^{0}(\overline(J))}$,derived base change gives,\\
\begin{equation*}
    \Phi_{\overline P}(F) \otimes^{L} k(\alpha) \cong R\Gamma(A, F \otimes \overline P_{\alpha})
\end{equation*}

This identity is crucial and will be used repeatedly to compute Fourier-Mukai transforms.
\subsection{Stability}
\begin{enumerate}
    \item Let $(X,H)$ be a projective variety of dimension $n$ with natural polarization $H$. For a torsion-free sheaf $F$ on $X$, the slope with respect to $H$ is,\\
\begin{equation*}
    \mu_{H}(F):= \frac{c_{1}(F).H^{n-1}}{rk(F)}
\end{equation*}
Let $F$ be a torsion free sheaf of rank $r$. Then $F$ is called slope stable with respect to $H$, if for all proper subsheaf $E$ of $F$ with $0<rk(E)<rk(F)$ we have $\mu_{H}(E)< \mu_{H}(F)$.

\item Let $X$ be a projective variety. A coherent sheaf $E$ of dimension $d$ is Gieseker-stable if $E$ is pure and for every proper subsheaf $F \subset E$ one has $p(F)< p(E)$. Here $p(E)$ is the reduced Hillbert polynomial of $E$.
\end{enumerate}

\section{Cohomological Vanishing and $WIT_{0}$ property}
In this section we analyze the Fourier-Mukai transforms of the push-forward of $a_{*}V$ ,where $V$ is a semistable torsion free sheaf $V$ on $X$. The goal is to establish uniform cohomological-vanishing on $X$ and derive the $WIT_{0}$ property.
\subsection{slope positivity and uniform $H^{1}$ vanishing}

Let $V$ be a torsion free sheaf on $X$. We say that $V$ is semistable if\\
\begin{equation*}
    \mu(W) \leq \mu(V) 
\end{equation*}
for all proper torsion free subsheaves $W \subset V$.
we assume throughout the section that,\\
\begin{equation*}
    \mu(V) > 2g-2
\end{equation*}
where $g=p_{a}(X)$
\begin{lemma}
If $V$ is semistable and $\mu(V)>2g-2$, then\\
\begin{equation*}
    H^{1}(X, V \otimes M)=0
\end{equation*}
for all $M \in \overline {Pic^{0}(X)}$
\end{lemma}
\begin{proof}
For $M \in \overline {Pic^{0}X}$, the sheaf $V \otimes M$ is still semistable and satisfies\\
\begin{equation*}
    \mu(V \otimes M)=\mu(V)
\end{equation*}
Here as $X$ is a singular curve with planner singularity it is local complete intersection in some projective space. Thus the cannonical sheaf is invertible. So using Serre duality we have,\\
\begin{equation*}
    H^{1}(X, V \otimes M) \cong Hom(V \otimes M, K)^{\vee}
\end{equation*}
Using Reiman-Roch theorem similarly we can compute that $K$ is an invertible sheaf with slope $2g-2$. As we have $\mu(V) >2g-2$, semistability implies\\
\begin{equation*}
    Hom(V \otimes M,K)=0
\end{equation*}
The claim follows.
\end{proof}
\begin{remark}
Note that the slope inequality $\mu(V) >2g-2$ is sharp for this argument. It ensures that $V$ admits no non-zero morphism to $K$, which is precisely the obstruction of $H^{1}$.
\end{remark}
\subsection{Derived fibers and $WIT_{0}$}
Let\\
\begin{equation*}
    F'=R^{i}\phi_{\overline P}(F)
\end{equation*}
 We now compute the derived fiber of $E$ at a point $\alpha \in \overline {(Pic^{0}(A)}$. By derived base change of Fourier-Mukai transform,\\
 \begin{equation*}
     E \otimes^{L} k(\alpha) \cong R\Gamma(A,a_{*}V \otimes \overline P_{\alpha})
 \end{equation*}
 using projection formula we obtain,\\
 \begin{equation*}
     E \otimes^{L} k(\alpha) \cong R\Gamma(X, V \otimes a^{*}P_{\alpha}) 
 \end{equation*} 
 now since $a^{*}(\overline P_{\alpha})$ is in the degree zero component so from the uniform vanishing lemma we get,\\
 \begin{equation*}
     H^{1}(C, V \otimes a^{*}\overline{P}_{\alpha})=0 
 \end{equation*}
 this holds for all $\alpha \in \overline {Pic^{0}A}$. Thus the derived fiber is concentrated in degree zero.
 \begin{proposition}
     Let $V$ be a torsion free sheaf on $X$ with $\mu(V) >2g-2$. Then:\\
     \begin{enumerate}
         \item $a_{*}V$ satisfies $WIT_{0}$ with respect to $\Phi_{\overline P}$;
         \item $E$ is torsion free with constant generic rank on $A$.
         \item $rk(E)=\chi(V)$
         
     \end{enumerate}
     \end{proposition}
     \begin{proof}
     Since for every $\alpha$ the derived fiber $E \otimes k(\alpha)$ is concentrated in degree $0$, we conclude that,\\
     \begin{equation*}
         R^{i}\Phi_{\overline P}(a_{*}V)=0 
     \end{equation*} for all $i>0$.
     Thus $a_{*}V$ satisfies $WIT_{0}$. Now in our situation we have $\overline P$ is flat with respect to projection and moreover,\\
     \begin{equation*}
         E \otimes k(\alpha) \cong h^{0}(X, V \otimes a^{*}\overline{P}_{\alpha})
     \end{equation*}
     Now $V$ is torsion free of fixed generic rank and each $\overline{P}_{\alpha}$ has generic rank $1$, So the dimension is constant outside the singular locus. Thus we can conclude that $E$ is torsion free of finite generic rank  and it is  precisely $\chi(V)$.
     \end{proof}
     \begin{corollary}
     If we have $\mu(V) >2g-2$ then,\\
     \begin{equation*}
         rk(E)=deg(V)+rk(V)(1-g)
     \end{equation*}
     \end{corollary}
     \begin{proof}
     Since $H^{1}(X,V)=0$ and Reiman-Roch theorem is valid for trosion free sheaves on nodal curve we get the desired identity.
     \end{proof}
     The results of this section show that positivity of slope on the curve translates into strong structural properties of Fourier-Mukai and on compactified Jacobian. The next section is devoted in the study of reducing this from $WIT_{0}$ to $IT_{0}$.
     \section{From $WIT_{0}$ to $IT_{0}$}
     In the previous section we proved that $E=\Phi_{\overline P}(a_{*}V)$ is generically locally free and satisfies $WIT_{0}$. We now produce an explicit torsion free sheaf on Jacobian which satisfies $IT_{0}$.
     \subsection{Reduction to Cohomology on the curve}
     Let $\alpha \in \overline{Pic^{0}(A)}$. By definition of the Fourier-Mukai transform and derived base change, we have\\
     \begin{equation*}
         E \otimes^{L} k(\alpha) \cong R\Gamma(A,a_{*}V \otimes \overline{P}_{\alpha}) 
     \end{equation*}
     using projection formula we get,\\
     \begin{equation*}
         E \otimes^{L} k(\alpha) \cong R\Gamma(X, V \otimes a^{*}\overline{P}_{\alpha})
     \end{equation*}
     Now taking cohomology we get that,\\
     \begin{equation*}
         H^{i}(A, E \otimes \alpha) \cong H^{i}(C, V \otimes a^{*}\overline P_{\alpha})
     \end{equation*}
     since the pull back of poincare sheaf belongs to the degree zero component of $X$, this reduces the $IT_{0}$ condition for $E$ to uniform vanishing statements on the curve $X$.
     Now are ready to state our main theorem\\
     \begin{theorem}
     Let $V$ be a semistable torsion free sheaf on $X$, with $\mu(V)>2g-2$. Then $E=\Phi_{\overline{P}}(a_{*}V)$ is a torsion free sheaf satisfying $IT_{0}$ property.
     \end{theorem}
     \begin{proof}
     As the cohomology for $i\geq 2$ vanishes for the curve $X$ we have to show that $H^{1}(C, V\otimes a^{*}\overline P_{\alpha})=0$. But this also follows from the lemma of uniform vanishing.
     \end{proof}

     \begin{corollary}
     The $IT_{0}$ sheaf constructed above is Giesker stable if $V$ is stable.
     \end{corollary}
     \begin{proof}
         In \cite{Dima} it is shown that the Fourier-Mukai Functor $\Phi_{\overline{P}}$ has a quasi inverse. Denote this by $\phi_{\overline{P}}^{-1}$. Now consider a destabibizing subsheaf $E'$ of $E$ with $p(E')>p(E)$. Now applying the Fourier-Mukai inverse, as it is a categorical equivalence we get $\Phi_{\overline{P}}^{-1}(E') \subset a_{*}V $. Lets call it $V'$. Again as Fourier- Mukai is a categorical equivalence it preserves cohomology. so we have that $p(V')\geq p(a_{*}V)$.  Now as $a_{*}V$ is supported on the curve itself $V'$ must also be supported on the curve. Thus we can identify $V'$ with a sheaf on $C$ and that sheaf must be subsheaf of $V$. So as $V$ is slope stable it is Gieseker-stable. Thus it is a contradiction. So $E$ is Gieseker stable if $V$ is stable.   
     \end{proof}
     \begin{corollary}
         if $V$ is stable then $E$ is Gieseker stable.
     \end{corollary}
     \begin{proof}
     In the previous result we have shown that if $V$ is stable then $E$ is Gieseker stable. Now over any projective scheme Gieseker stable bundles are simple\cite[corollary 1.2.8]{Lehn}. Thus $E$ is simple. However simplicity of $E$ can be shown independently also using that $V$ is simple and $\phi_{\overline{P}}$ is a categorical equivalence.
     \end{proof}
     
      \begin{remark}
         Note that ,without loss of generality we can assume the bundle $V$ to be stable, as from Jordan-Holder filtration of of semistable bundles ,we can always produce a stable bundle of same slope which arise as quotient from the filtration.
      \end{remark}
      \begin{remark}
      All the constructions we have done here essentially holds in the smooth curve case. There we get a locally free sheaf after Fourier-Mukai transform and we can work with Jacobian itself which is a abelian variety. We hope that the existence of stable simple $IT_{0}$ torsion free sheaf of fixed rank will be helpful to study the moduli-space of them on compcatified jacobian and in smooth case on jacobian.
      \end{remark}

\end{document}